\documentclass[12pt, reqno]{amsart}
\usepackage{esint,amsfonts,amsmath,amssymb,epsfig,stmaryrd,mathrsfs,bm,accents,yfonts,mathtools,graphicx,pgf}
\usepackage{a4wide}

\begingroup

\newtheorem{theorem}{Theorem}[section]
\newtheorem{lemma}[theorem]{Lemma}
\newtheorem{propos}[theorem]{Proposition}

\newtheorem{claim}{Claim}
\endgroup

\theoremstyle{definition}
\newtheorem{definition}[theorem]{Definition}

\newtheorem{remark}[theorem]{Remark}

\numberwithin{equation}{section}

\newcommand{\eps}{{\varepsilon}}

\newcommand{\Lip}{{\text {Lip}}}

\newcommand\weaks{{\rightharpoonup^*}\,}
\newcommand\weak{{\rightharpoonup}\,}

\newcommand\res{\mathop{\hbox{\vrule height 7pt width .5pt depth 0pt
\vrule height .5pt width 6pt depth 0pt}}\nolimits}

\newcommand{\cG}{{\mathcal{G}}}

\newcommand{\cL}{{\mathcal{L}}}
\newcommand{\cH}{{\mathcal{H}}}

\newcommand{\Pe}{{\mathscr{P}}}

\def\R#1{{\mathbb R}^{#1}}

\newcommand\N{{\mathbb N}}

\newcommand{\Om}{\Omega}
\def\I#1{{\mathcal{A}}_{#1}}

\newcommand{\Iqs}{{\mathcal{A}}_Q(\R{n})}
\newcommand{\Iq}{{\mathcal{A}}_Q}
\def\a#1{\left\llbracket{#1}\right\rrbracket}
\newcommand{\abs}[1]{\left|#1\right|}
\newcommand{\norm}[2]{\left\|#1\right\|_{#2}}
\newcommand{\D}{\textup{Dir}}
\newcommand{\de}{\partial}

\newcommand{\ph}{\varphi}
\newcommand{\ra}{\right\rangle}
\newcommand{\la}{\left\langle}

\title[Lower semicontinuous $Q$-functionals]{Lower semicontinuous functionals for Almgren's multiple valued functions}
\author[De Lellis, Focardi and Spadaro]{Camillo De Lellis, Matteo Focardi and Emanuele Nunzio Spadaro}

\address{Universit\"at Z\"urich}
\email{camillo.delellis@math.uzh.ch}

\address{Universit\`a di Firenze}
\email{focardi@math.unifi.it}

\address{Hausdorff Center for Mathematics of Bonn}
\email{emanuele.spadaro@hcm.uni-bonn.de}

\begin{document}

\begin{abstract}
We consider general integral functionals on the Sobolev spaces
of multiple valued functions introduced by Almgren. 
We characterize the semicontinuous ones and recover earlier results
of Mattila in \cite{Mat} as a particular case. Moreover, we answer positively to
one of the questions raised by Mattila in the same paper.
\end{abstract}

\maketitle

\section{Introduction}\label{s:intro}
In his big regularity paper \cite{Alm}, Almgren developed a new theory of weakly differentiable multiple valued maps minimizing a suitable
generalization of the classical Dirichlet energy.
He considered maps defined on a Lipschitz domain $\Omega\subset\R{m}$ and taking 
values in the space of $Q$ unordered points of $\R{n}$,
which minimize the integral of the squared norm of the derivative 
(conveniently defined).
The regularity theory for these so called $\D$-minimizing $Q$-valued maps 
is a cornerstone in his celebrated proof that the Hausdorff dimension of the
singular set of an $m$--dimensional area-minimizing current is at most $(m-2)$.

The existence of $\D$-minimizing functions with prescribed boundary data is proven in \cite{Alm} 
via the direct method in the calculus of variations. Thus,
the generalized Dirichlet energy is semicontinuous under weak convergence.
This property is not specific of the energy considered by Almgren.
Mattila in \cite{Mat} considered some energies induced by homogeneous quadratic polynomials 
of the partial derivatives. His energies are the first non-constant
term in the Taylor expansion of elliptic geometric integrands and hence generalize
Almgren's Dirichlet functional, which is the first non-constant term in 
the expansion of the area functional.

Mattila showed that these quadratic functionals are lower semicontinuous under weak convergence.
A novelty in Mattila's work was the impossibility to use Almgren's extrinsic biLipschitz 
embeddings of the space of $Q$-points into
a Euclidean space, because of the more complicated form of the energies (cp. with \cite{Alm}
and \cite{DLSp1} for the existence and properties of these embeddings).
In this paper we push forward the investigation of Mattila and, taking advantage of the intrinsic 
metric theory for $Q$-valued functions
developed in \cite{DLSp1}, we generalize his results 
to the case of general integral functionals defined
on Sobolev spaces of $Q$-functions. We obtain a complete characterization of the semicontinuity
and a simple criterion to recognize efficiently a specific class of semicontinuous functionals.
Mattila's $Q$-semielliptic energies fall obviously into this class.
Indeed, a simple corollary of our analysis is that a quadratic energy as considered in \cite{Mat}
is $Q$-semielliptic if and only if it is quasiconvex (see Definition \ref{d:qc} and Remark \ref{r: mat}
for the relevant definitions).
Moreover, in the special cases of dimensions $m=2$ or $n=2$, we can answer positively to the question posed by Mattila himself
on the equivalence of $Q$-semiellipticity and $1$-semiellipticity.

\subsection{Quasiconvexity and lower semicontinuity}
In order to illustrate the results, we introduce the following terminology
(we refer to \cite{DLSp1} and Subsection~\ref{ss:q} for the relevant definitions and terminology concerning
$Q$--valued maps).

Let $\Omega\subset\R{m}$ be a bounded open set.
A measurable map $f:\Omega\times \left( \R{n}\right)^Q \times \left(\R{m\times n}\right)^Q\to \R{}$
is called a {\em $Q$-integrand} if, for every permutation $\pi$ of $\{1,\ldots,Q\}$,
\[
f (x,a_1, \ldots, a_Q, A_1, \ldots, A_Q) = f (x,a_{\pi (1)}, \ldots, a_{\pi (Q)}, A_{\pi(1)}, \ldots, A_{\pi (Q)}).
\]
Note that, by \eqref{e:first order approx} (see also \cite[Remark 1.11]{DLSp1}), given a weakly differentiable $Q$-valued map $u$, 
the expression $f(\cdot,u,Du)=f(\cdot,u_1,\ldots,u_Q,Du_1,\ldots,Du_Q)$ 
is well defined almost everywhere in $\Omega$.
Thus, for any Sobolev $Q$-valued function the following energy makes sense:
\begin{equation}\label{e:F}
F(u)=\int_\Omega f\big(x,u(x),Du(x)\big)dx.
\end{equation}

Our characterization of (weakly) lower-semicontinuous functionals $F$ 
is the counterpart of Morrey's celebrated 
result in the vectorial calculus of the variations (see \cite{MorPaper},
\cite{MorBook}).
We start by introducing the relevant notion of quasiconvexity, which is
a suitable generalization of Morrey's definition. From now on we set $C_r:=[-r/2,r/2]^m$.

\begin{definition}[Quasiconvexity]\label{d:qc}
Let $f: \left(\R{n}\right)^Q\times\left(\R{m\times n}\right)^Q \to \R{}$ be a locally bounded $Q$-integrand.
We say that $f$ is \textit{quasiconvex} if the following holds
for every affine $Q$-valued function
$u (x)= \sum_{j=1}^J q_j \a{a_j + L_j \cdot x}$, with $a_i\neq a_j$ for $i\neq j$.
Given any collection of maps $w^j\in W^{1,\infty} (C_1,\I{q_j})$
with $w^j|_{\de C_1} = q_j \a{a_j +L_j|_{\de C_1}}$ we have the inequality
\begin{equation}\label{e:qc}
f\big(u(0),Du(0)\big)\leq \int_{C_1} f\big(\underbrace{a_1, \ldots, a_1}_{q_1},
\ldots,\underbrace{a_J, \ldots, a_J}_{q_J},Dw^1(x), \ldots, Dw^J(x)\big)dx.
\end{equation}
\end{definition}

The main result is the following.

\begin{theorem}\label{t:qc}
Let $p\in [1, \infty[$ and $f:\Omega\times \left( \R{n}\right)^Q \times 
\left(\R{m\times n}\right)^Q\to \R{}$ be a continuous $Q$-integrand.
If $f (x, \cdot, \cdot)$ is quasiconvex for every $x\in \Omega$ and
\[
0 \leq f (x, a, A)\leq C (1+ |a|^q+|A|^p) \qquad \mbox{for some constant $C$},
\]
where $q=0$ if $p>m$, $q=p^*$ if $p<m$ and $q\geq 1$ is any exponent if $p=m$,
then the functional $F$ in \eqref{e:F} is weakly lower semicontinuous in 
$W^{1,p} (\Omega, \Iqs)$.
Conversely, if $F$ is weakly$^*$ lower semicontinuous in $W^{1,\infty} (\Omega, \Iqs)$,
then $f(x, \cdot, \cdot)$ is quasiconvex for every $x\in \Omega$.
\end{theorem}

\begin{remark}\label{r: mat2}
It is easy to see that a quadratic integrand is $Q$-semielliptic in the sense of Mattila if and only if it is
quasiconvex, cp.~to Remark \ref{r: mat}.
\end{remark}

\subsection{Polyconvexity} We continue to follow the classical path of the vectorial
calculus of variations and introduce 
a suitable generalization of the well-known notion of polyconvexity (see \cite{MorBook},
\cite{Ball}).
Let $N:=\min \{m,n\}$, $\tau(n,m):=\sum_{k=1}^N\binom{m}{k}\binom{n}{k}$ and 
define $M:\R{n\times m}\to\R{\tau(m,n)}$ as
$M(A):=\big(A,\mathrm{adj}_2A,\ldots,\mathrm{adj}_NA\big)$,
where $\mathrm{adj}_kA$ stands for 
the matrix of all $k\times k$ minors of $A$.

\begin{definition}\label{d:pc}
A $Q$-integrand $f:\left(\R{n}\right)^Q\times \left(\R{n\times m}\right)^Q \to \R{}$ is \textit{polyconvex} if there exists a map
$g:\left(\R{n}\right)^Q\times \left(\R{\tau(m,n)}\right)^Q\to \R{}$ such that: 
\begin{itemize}
\item[(i)] the function $g(a_1, \ldots, a_Q,\cdot):\left(\R{\tau(m,n)}\right)^Q\to \R{}$ is convex for every $a_1, \ldots, a_Q \in \R{n}$,
\item[(ii)] for every $a_1, \ldots, a_Q \in \R{n}$ and $(L_1,\ldots,L_Q)\in(\R{n\times m})^Q$ it holds
\begin{equation}\label{e:char-cvx}
f\big(a_1, \ldots, a_Q,L_1,\ldots, L_Q\big)=g\big(a_1, \ldots, a_Q,M(L_1),\ldots, M(L_Q)\big).
\end{equation}
\end{itemize}
\end{definition}

Polyconvexity is much easier to verify.
For instance, if $\min\{m,n\}\leq 2$, quadratic integrands are polyconvex if and only if they are $1$-semielliptic
in the sense of Mattila, cp.~to Remark \ref{r: mat3}.
Combining this with Remark \ref{r: mat2} and Theorem \ref{t:pc=>qc}, we easily
conclude that $Q$-semiellipticity and $1$-semiellipticity coincide in this case, as suggested
by Mattila himself in \cite{Mat}.

\begin{theorem}\label{t:pc=>qc}
Every locally bounded polyconvex $Q$-integrand $f$ is $Q$-quasiconvex.
\end{theorem}

For integrands on single valued maps,
the classical proof of Theorem \ref{t:pc=>qc}
relies on suitable integration by parts formulas, called Piola's identities
by some authors. These identities can be shown by direct computation. However, an elegant
way to derive them is to rewrite the quantities involved as integrals of suitable
differential forms over the graph of the given map. 
The integration by parts is then explained via
Stokes' Theorem. This point of view is the starting of the theory of
Cartesian currents developed by Giaquinta, Modica and Sou\v{c}ek (see the
monograph \cite{GMS1,GMS2}). Here we take this approach to derive 
similar identities in the case of $Q$-valued maps, building on the obvious 
structure of current induced by the graph
of Lipschitz $Q$-valued maps $f: \Omega\to
\Iqs$ (which we denote by ${\rm gr}\, (f)$).
A key role is played by the intuitive identity $\partial\, {\rm gr}\, (f) = 
{\rm gr} \left(f|_{\partial \Omega}\right)$, which for $Q$-valued maps is less
obvious. A rather lengthy proof of this fact was given for the first time
in \cite{Alm}. We refer to Appendix C of \cite{DLSp2} for a much shorter derivation.
A final comment is in order. Due to the combinatorial complexity of $Q$-valued maps,
we do not know whether Theorem \ref{t:pc=>qc} can be proved without using the theory
of currents. 

The paper is organized in three sections. The first one contains 
three technical lemmas on $Q$-valued Sobolev functions, proved using
the language of \cite{DLSp1} (which differs slightly from Almgren's original
one).
In Section \ref{s:quasi} we prove Theorem \ref{t:qc} and in 
Section \ref{s:poly} Theorem \ref{t:pc=>qc}.
In the appendix we collect some results on equi-integrable functions,
essentially small variants of Chacon's biting lemma, which have already appeared
in the literature: we include their proofs for
reader's convenience.

\section{$Q$-valued functions}\label{s:prel}
In this section we recall the notation and terminology of \cite{DLSp1},
and provide some preliminary results which will be 
used in the proofs of Theorem \ref{t:qc} and Theorem \ref{t:pc=>qc}.

\subsection{Sobolev $Q$-valued functions}\label{ss:q}
$Q$-valued functions are maps valued in the complete metric space of unordered sets of $Q$ points in $\R{n}$.

\begin{definition}\label{d:IQ}  
We denote by $(\Iqs, \cG)$ the metric space of unordered $Q$-tuples given by
\begin{equation*}
\Iqs :=\left\{\sum_{i=1}^Q\a{P_i}\,:\,P_i\in\R{n}\;\textrm{for every  }i=1,\ldots,Q\right\},
\end{equation*}
where $\a{P_i}$ denotes the Dirac mass in $P_i\in \R{n}$ and
\begin{equation*}
\cG(T_1,T_2)\;:=\;\min_{\sigma\in\Pe_Q}\sqrt{\sum_i\abs{P_i-S_{\sigma(i)}}^2},
\end{equation*}
with $T_1=\sum_i\a{P_i}$ and $T_2=\sum_i\a{S_i}\in \Iqs$, and $\Pe_Q$ denotes the group of permutations of $\left\{1,\ldots,Q\right\}$.
\end{definition}

Given a vector $v\in\R{n}$, we denote by $\tau_v(T)$ the translation of the $Q$-point $T=\sum_i\a{T_i}$ under $v$ given by
\begin{equation}\label{e:translation}
\tau_v(T):=\sum_i\a{T_i-v}.
\end{equation}
Continuous, Lipschitz, H\"older and (Lebesgue) measurable 
functions from $\Omega$ into $\Iq$ are defined in the usual way.
It is a general fact that any measurable $Q$-valued function $u: \Omega\to \Iq$ can be written as the ``sum'' of $Q$ measurable functions $u_1,\ldots,u_Q$ \cite[Proposition~0.4]{DLSp1}:
\[
u(x)=\sum_i \a{u_i(x)}\quad\text{for a.e. }x\in \Omega.
\]

We now recall the definition of the Sobolev spaces of functions taking values in the metric space of $Q$-points.

\begin{definition}\label{d:W1p}
A measurable $u:\Omega\rightarrow\Iq$ is in the Sobolev class
$W^{1,p}$ ($1\leq p\leq\infty$) if there exists $\varphi\in L^p(\Omega;[0,+\infty))$ such that 
\begin{itemize}
\item[(i)] $x\mapsto\cG (u(x),T)\in W^{1,p}(\Omega)$ for all $T\in \Iq$;
\item[(ii)] $\abs{D\, \cG (u, T)}\leq\varphi$ a.e. in $\Omega$ for all $T\in \Iq$.
\end{itemize}
\end{definition}
As for classical Sobolev maps, an important feature of Sobolev $Q$-valued functions is the existence of the approximate differential almost everywhere.
Given $u\in W^{1,p}(\Omega,\Iqs)$, there exists a $Q$ map $Du=\sum_i\a{Du_i}:\Omega\to \Iq(\R{m\times n})$ such that, for almost every $x_0\in \Omega$, the first order approximation
\begin{equation}\label{e:first order approx}
T_{x_0} u(x)\;:=\;\sum_i\a{Du_i(x_0)\cdot(x-x_0)+u_i(x_0)}
\end{equation}
satisfies the following:
\begin{itemize}
\item[(i)] there exists a set $\tilde \Omega$ with density one at $x_0$ such that $\cG(u(x),T_{x_0} u)=o(\abs{x-x_0})$ as $x\to x_0$, $x\in\tilde\Omega$;
\item[(ii)] $Du_i(x_0)=Du_j(x_0)$ if $u_i(x_0)=u_j(x_0)$.
\end{itemize}
Moreover, the map $Du$ is $L^p$ integrable, meaning that
\[
|Du|:=\sqrt{\sum_i |Du_i|^2}\in L^p(\Omega).
\]
Finally, we recall the definition of weak convergence in $W^{1,p}(\Omega,\Iqs)$.

\begin{definition}\label{d:weak convergence}
Let $u_k, u\in W^{1,p}(\Om;\Iq)$. We say that $u_k$ converges weakly 
to $u$ for $k \to \infty$, (and we write $u_k\rightharpoonup u$) in $W^{1,p}(\Om;\Iq)$,
if
\begin{itemize}
\item[(i)] $\int\cG(f_k,f)^p\to0$, for $k\to\infty$;
\item[(ii)] $\sup_{k}\int |Df_k|^p<\infty$.
\end{itemize}
\end{definition}

\subsection{$L^p$-approximate differentiability}
Here we prove a more refined differentiability result.
\begin{lemma}\label{l:caldzyg}
Let $u\in W^{1,p}(\Omega,\Iq)$. Then, for $\cL^m$-a.e.~$x_0\in\Omega$ it holds
\begin{equation}\label{e:caldzyg}
\lim_{\rho\to0}\rho^{-p-m}\int_{C_\rho(x_0)}\cG^p(u,T_{x_0}u)=0.
\end{equation}
\end{lemma}
\begin{proof}
By the Lipschitz approximation in \cite[Proposition 4.4]{DLSp1}, there exists a family of functions $(u_\lambda)$ such that:
\begin{itemize}
 \item[(a)] $\Lip(u_\lambda)\leq \lambda$ and $d_{W^{1,p}}(u,u_\lambda)=o(1)$ as $\lambda\to+\infty$;
 \item[(b)] the sets $\Omega_\lambda=\{x: T_{x}u=T_{x}u_\lambda\}$ satisfy $\Omega_\lambda\subset \Omega_{\lambda'}$
 for $\lambda<\lambda'$ and $\cL^m(\Omega\setminus \Omega_\lambda)=o(1)$ as $\lambda\to+\infty$.
\end{itemize}
We prove \eqref{e:caldzyg} for the points $x_0\in \Omega_\lambda$ which are Lebesgue points for $\chi_{\Omega_\lambda}$ and
$|Du|^p  \chi_{\Omega\setminus \Omega_\lambda}$, for some $\lambda\in\N$, that is
\begin{equation}\label{e:lebesgue}
\lim_{\rho\to0}\fint_{C_\rho(x_0)}\chi_{\Omega_\lambda}=1
\quad\text{and}\quad
\lim_{\rho\to0}\fint_{C_\rho(x_0)}|Du|^p\chi_{\Omega\setminus \Omega_\lambda}=0.
\end{equation}
Let, indeed, $x_0$ be a point as in \eqref{e:lebesgue} for a fixed $\Omega_\lambda$. Then,
\begin{align}\label{e:caldzyg2}
\fint_{C_\rho(x_0)}\cG^p(u,T_{x_0}u)&\leq
2^{p-1}\fint_{C_\rho(x_0)}\cG^p(u_\lambda,T_{x_0}u_\lambda)+
2^{p-1}\fint_{C_\rho(x_0)}\cG^p(u_\lambda,u)\notag\\
&\leq o(\rho^p)+
C\rho^{p-m}\int_{C_\rho(x_0)\setminus \Omega_\lambda}|D(\cG(u_\lambda,u))|^p,
\end{align}
where in the latter inequality we used Rademacher's theorem for $Q$-functions (see \cite[Theorem 1.13]{DLSp1}) 
and a Poincar\'e inequality for the classical Sobolev function $\cG(u,u_\lambda)$ which by \eqref{e:lebesgue} satisfies
$$
\Omega_\lambda\subseteq \big\{\cG(u,u_\lambda)=0\big\}
\quad\text{and}\quad
\rho^{-m}\cL^m(C_\rho(x_0)\cap \Omega_\lambda)\geq 1/2\quad\text{for small }\;\rho.
$$
Since $\cG(u,u_\lambda)=\sup_{T_i}|\cG(u,T_i)-\cG(T_i,u_\lambda)|$ and
$$
D|\cG(u,T_i)-\cG(T_i,u_\lambda)|\leq |D\cG(u,T_i)|+|D\cG(T_i,u_\lambda)|\leq |Du|+|Du_\lambda|\quad
\cL^m\text{-a.e. on }\;\Omega,
$$
we conclude (recall that $\lambda\leq C|Du|$ on  $\Omega\setminus \Omega_\lambda$)
\begin{align*}
\rho^{p-m}\int_{C_\rho(x_0)\setminus \Omega_\lambda}|D(\cG(u,u_\lambda))|^p&\leq 
\rho^{p-m}\int_{C_\rho(x_0)\setminus \Omega_\lambda}\sup_i \big(D|\cG(u,T_i)-\cG(T_i,u_\lambda)|\big)^p\\
&\leq C\rho^{p-m}\int_{C_\rho(x_0)\setminus \Omega_\lambda}
|Du|^p\stackrel{\eqref{e:lebesgue}}{=}o(\rho^p),
\end{align*}
which finishes the proof.
\end{proof}

\subsection{Equi-integrability}
In the first lemma we show how a weakly convergent sequence of $Q$-functions can be truncated in order to obtain
an equi-integrable sequence still weakly converging to the same limit.
This result is the analog of \cite[Lemma 2.3]{FMP} for $Q$-valued functions and constitute 
a main point in the proof of the sufficiency of quasiconvexity for the lower semicontinuity.
Details on equi-integrability can be found in the Appendix.

\begin{lemma}\label{l:equiint}
Let $(v_k)\subset W^{1,p}(\Omega,\Iq)$ be weakly converging to $u$.
Then, there exists a subsequence $(v_{k_j})$ and a sequence $(u_j)\subset
W^{1,\infty}(\Omega,\Iq)$ such that
\begin{itemize}
\item[(i)] $\cL^m(\{v_{k_j}\neq u_j\})=o(1)$ and $u_j\weak u$ in
$W^{1,p}(\Omega,\Iq)$;
\item[(ii)] $(|Du_j|^p)$ is equi-integrable;
\item[(iii)] if $p\in[1,m)$, $(|u_j|^{p^*})$ is equi-integrable and,
if $p=m$, $(|u_j|^q)$ is equi-integrable for any $q\geq 1$.
\end{itemize}
\end{lemma}
\begin{proof}
Let $g_k:=M^p(|Dv_k|)$ and notice that, by the estimate on the maximal function operator (see \cite{St} for instance),
$(g_k)\subset L^1(\Omega)$ is a bounded sequence.
Applying Chacon's biting lemma (see Lemma~\ref{l:BM} in the Appendix) to $(g_k)$, we get a subsequence $(k_j)$ and a sequence
$t_j\nearrow+\infty$ such that $(g_{k_j}\wedge t_j)$ are equi-integrable.

Let $\Omega_j:=\{x\in\Omega:\, g_{k_j}(x)\leq t_j\}$ and $u_j$ be the Lipschitz extension of $v_{k_j}\vert_{\Omega_j}$
with Lipschitz constant $c\,t_j^{1/p}$ (see \cite[Theorem 1.7]{DLSp1}).
Then, following \cite[Proposition 4.4]{DLSp1}, it is easy to verify that $\cL^m(\Omega\setminus\Omega_j)=o(t_j^{-1})$ and
$d_{W^{1,p}}(u_j,v_{k_j})=o(1)$.
Thus, (i) follows immediately from these properties and (ii) from
$$
|Du_j|^p=|Dv_{k_j}|^p\leq g_{k_j}\wedge t_j \text{ on } \Omega_j
\quad\text{and}\quad
|Du_j|^p\leq c\, t_j=c\, (g_{k_j}\wedge t_j) \text{ on } \Omega\setminus\Omega_j.
$$
As for (iii), note that the functions $f_j:=\cG(u_j,Q\a{0})$ are
in $W^{1,p}(\Omega)$, with $|Df_j|\leq|Du_j|$ by the very definition of metric space valued Sobolev maps.
Moreover, by (i), $f_j$ converge weakly to $|u|$, since $\||u|-f_j\|_{L^p}\leq \|\cG(u,u_j)\|_{L^p}$.
Hence, $(|f_j|^p)$ and $(|Df_j|^p)$ are equi-integrable.
In case $p\in[1,m)$, this implies (see Lemma~\ref{l:equipstar}) the equi-integrability of $(|u_j|^{p^*})$.
In case $p=m$, the property follows from H\"older inequality and Sobolev embedding (details are left to the reader).
\end{proof}

\subsection{Averaged equi-integrability}
The next lemma gives some properties of 
sequences of functions whose blow-ups are equi-integrable.
In what follows a function $\ph:[0,+\infty]\to[0,+\infty]$ is said superlinear at infinity if
$\lim_{t\uparrow+\infty}\frac{\ph(t)}{t}=+\infty$.

\begin{lemma}\label{l:equi mean}
Let $g_k\in L^1(\Omega)$ with $g_k\geq0$ and $\sup_k\fint_{C_{\rho_k}}\ph(g_k)<+\infty$,
where $\rho_k\downarrow 0$ and $\ph$ is superlinear at infinity.
Then, it holds
\begin{equation}\label{e:eq mean1}
\lim_{t\to+\infty}\left(\sup_k\rho_k^{-m}\int_{\{g_k\geq t\}}g_k\right)=0
\end{equation}
and, for sets $A_k\subseteq C_{\rho_k}$ such that $\cL^m(A_k)=o(\rho_k^{m})$,
\begin{equation}\label{e:eq mean2}
\lim_{k\to+\infty}\rho_k^{-m}\int_{A_k}g_k=0.
\end{equation}
\end{lemma}
\begin{proof}
Using the superlinearity of $\ph$, for every $\eps>0$ 
there exists $R>0$ such that $t\leq \eps\ph(t)$ for every $t\geq R$,
so that
\begin{equation}\label{e:eq mean3}
\limsup_{t\to+\infty}\left(\sup_k\rho_k^{-m}\int_{\{g_k\geq t\}}g_k\right)\leq
\eps\, \sup_k\fint_{C_{\rho_k}}\ph(g_k)\leq C\,\eps.
\end{equation}
Then, \eqref{e:eq mean1} follows as $\eps\downarrow0$.
For what concerns \eqref{e:eq mean2}, we have
\begin{align*}
\rho_k^{-m}\int_{A_k}g_k&=\rho_k^{-m}\int_{A_k\cap\{g_k\leq t\}}g_k+\rho_k^{-m}\int_{A_k\cap\{g_k\geq t\}}g_k
\leq
t\rho_k^{-m}\cL^m(A_k)+\sup_k\rho_k^{-m}\int_{\{g_k\geq t\}}g_k.
\end{align*}
By the hypothesis $\cL^m(A_k)=o(\rho_k^{m})$,
taking the limit as $k$ tends to $+\infty$ and then as $t$ tends to
$+\infty$, by \eqref{e:eq mean1} the right hand side above vanishes.
\end{proof}

\subsection{Push-forward of currents under $Q$-functions}\label{s:current}
We define now the integer rectifiable current associated to the graph of a $Q$-valued function.
As for Lipschitz single valued functions, we can associate to the graph of a Lipschitz $Q$-function
$u:\Omega\to\Iq$ a rectifiable current $T_{u,\Omega}$ defined by
\begin{equation}\label{e:Tu}
\la T_{u,\Omega}, \omega\ra = \int_{\Omega} \sum_i\big\langle \omega\left(x,u_i(x)\right), \vec{T}_{u_i}(x)\big\rangle\,d\,\cH^m(x)
\quad\forall\;\omega\in\mathscr{D}^m(\R{m+n}),
\end{equation}
where $\vec{T}_{u_i}(x)$ is the $m$-vector given by
$\left(e_1+\de_1 u_{i}(x)\right)\wedge\cdots\wedge\left(e_m+\de_m u_{i}(x)\right)\in \Lambda_m(\R{m+n}).$
In coordinates, writing $\omega(x,y)=\sum_{l=1}^N\sum_{|\alpha|=|\beta|=l}\omega^l_{\alpha\beta}(x,y)dx_{\bar\alpha}\wedge dy_{\beta}$,
where $\bar \alpha$ denotes the complementary multi-index of $\alpha$, the current $T_{u,\Omega}$ acts in the following way:
\begin{equation}\label{e:Tu minori}
\la T_{u,\Omega},\omega\ra =
\int_{\Omega}\sum_{i=1}^Q \sum_{l=1}^N\sum_{|\alpha|=|\beta|=l}
\sigma_{\alpha}\,\omega^l_{\alpha\beta}\big(x,u_i(x)\big)M_{\alpha\beta}\big(Du_i(x)\big)dx,
\end{equation}
with $\sigma_{\alpha}\in\{-1,1\}$ the sign of the permutation ordering $(\alpha,\bar\alpha)$ in the natural increasing order and
$M_{\alpha\beta}(A)$ denoting the $\alpha, \beta$ minor of a matrix $A\in\R{n\times m}$,
$$
M_{\alpha \beta} (A) := {\rm det}\, \left(\begin{array}{clc}
A_{\alpha_1\beta_1} & \dots & A_{\alpha_1\beta_k}\\
\vdots & \ddots & \vdots\\
A_{\alpha_k\beta_1} & \dots & A_{\alpha_k\beta_k}
\end{array}\right).
$$

Analogously, assuming that $\Omega$ is a Lipschitz domain, using parametrizations of the boundary,
one can define the current associate to the graph of $u$ restricted to $\de \Omega$, and both $T_{u,\Omega}$ and $T_{u,\de\Omega}$ turn out to be rectifiable current -- see \cite[Appendix C]{DLSp2}.
The main result about the graphs of Lipschitz $Q$-functions we are going to use is the following theorem proven in \cite[Theorem C.3]{DLSp2}.

\begin{theorem}\label{t:de Tf}
For every $\Omega$ Lipschitz domain and $u\in \Lip(\Omega,\Iq)$, $\de \,T_{u,\Omega}=T_{u,\de \Omega}$.
\end{theorem}

\section{Quasiconvexity and lower semicontinuity}\label{s:quasi}
In this section we prove Theorem \ref{t:qc}.
Before starting, we link our notion of quasiconvexity with the $Q$-semiellipticity introduced in \cite{Mat}.

\begin{remark}\label{r: mat}
Following Mattila, a quadratic integrand is a function of the form
\begin{equation*}
E(u):=\int_\Omega \sum_i\langle A Du_i,Du_i\rangle,
\end{equation*}
where $\R{n\times m}\ni M\mapsto A\,M\in \R{n\times m}$ is a linear symmetric map.
This integrand is called $Q$-semielliptic if
\begin{equation}
\int_{\R{m}} \sum_i\langle A Df_i,Df_i\rangle\geq 0 \quad \forall \;f\in\Lip(\R{m},\Iq)\;\text{with compact support.}
\end{equation}
Obviously a $Q$-semielliptic quadratic integrand is $k$-semielliptic for every $k\leq Q$.
We now show that $Q$-semiellipticity and quasiconvexity coincide.
Indeed, consider a linear map $x\mapsto L\cdot x$ and a Lipschitz $k$-valued function $g(x)=\sum_{i=1}^k\a{f_i(x)+L\cdot x}$,
where $f=\sum_i\a{f_i}$ is compactly supported in $C_1$ and $k\leq Q$.
Recall the notation $\eta\circ f=k^{-1}\sum_i f_i$ and the chain rule formulas in \cite[Section 1.3.1]{DLSp1}.
Then,
\begin{align*}
E(g)&=E(f)+k\,\langle A\,L,L\rangle+2\int_{C_1}\sum_i\langle A\,L,Df_i\rangle\\
&=E(f)+k\,\langle A\,L,L\rangle+2\,k\int_{C_1}\langle A\,L,D(\eta\circ f)\rangle=
E(f)+k\,\langle A\,L,L\rangle,
\end{align*}
where the last equality follows integrating by parts.
This equality obviously implies the equivalence of $Q$-semiellipticity and quasiconvexity.
\end{remark}

\subsection{Sufficiency of quasiconvexity}
We prove that, given a sequence $(v_k)\subset W^{1,p}(\Omega,\Iq)$ weakly converging to $u\in W^{1,p}(\Omega,\Iq)$ and
$f$ as in the statement of Theorem \ref{t:qc}, then
\begin{equation}\label{e:sc}
F(u)\leq \liminf_{k\to\infty}F(v_k).
\end{equation}

Up to extracting a subsequence, we may assume that the inferior limit in \eqref{e:sc} is actually a limit
(in what follows, for the sake of convenience, subsequences will never be relabeled).
Moreover, using Lemma \ref{l:equiint}, again up to a subsequence, there exists $(u_k)$ such that (i)-(iii) in Lemma \ref{l:equiint}
hold.
If we prove
\begin{equation}\label{e:sc2}
F(u)\leq \lim_{k\to\infty}F(u_k),
\end{equation}
then \eqref{e:sc} follows, since, by the equi-integrability properties (ii) and (iii),
\begin{align*}
F(u_k)&=\int_{\{v_k=u_k\}}f(x,v_k,Dv_k)+
\int_{\{v_k\neq u_k\}}f(x,u_k,Du_k)\nonumber\\
&\leq F(v_k)+C\int_{\{v_k\neq u_k\}}\left(1+|u_k|^q+|Du_k|^p\right)
=F(v_k)+o(1).
\end{align*}

For the sequel, we will fix a function $\varphi:[0,+\infty)\to[0,+\infty]$ superlinear at infinity such that
\begin{equation} \label{e:char2bis}
\sup_k\int_{\Om}\big(\varphi(|u_k|^q)+\varphi(|Du_k|^p)\big)dx<+\infty.
\end{equation}
In order to prove \eqref{e:sc2}, it suffices to show that there exists a subset of full measure $\tilde \Omega\subseteq\Omega$ such
that for $x_0\in \tilde \Omega$ we have
\begin{equation}\label{e:sc density}
f(x_0,u(x_0),Du(x_0))\leq\frac{d\mu}{d\cL^m}(x_0),
\end{equation}
where $\mu$ is the weak$^*$ limit in the sense of measure of any converging subsequence of $\big(f(x,u_k,Du_k)\cL^m\res\Omega\big)$.
We choose $\tilde \Omega$ to be the set of points $x_0$ which satisfy \eqref{e:caldzyg} in Lemma~\ref{l:caldzyg} and,
for a fixed subsequence with $\big(\ph(|u_k|^q)+\ph(|Du_k|^p)\big)\cL^m\res\Omega\weaks\nu$, satisfy
\begin{equation}\label{e:pt1}
\frac{d\nu}{d\cL^m}(x_0)<+\infty.
\end{equation}
Note that such $\tilde\Omega$ has full measure by the standard Lebesgue differentiation theory of
measure and Lemma~\ref{l:caldzyg}.

We prove \eqref{e:sc density} by a blow-up argument following Fonseca and M\"uller \cite{FoMu}.
Since in the space $\Iq$ translations make sense only for $Q$ multiplicity points,
blow-ups of $Q$-valued functions are not well-defined in general.
Hence, to carry on this approach, we need first to decompose the approximating
functions $u_k$ according to the structure of the first order-approximation
$T_{x_0}u$ of the limit, in such a way to reduce to the case of full multiplicity tangent planes.

\begin{claim}\label{cl:1}
Let $x_0\in\tilde \Omega$ and $u(x_0)=\sum_{j=1}^Jq_j\,\a{a_j}$, with $a_i\neq a_j$ for $i\neq j$.
Then, there exist $\rho_k\downarrow 0$ and $(w_k)\subseteq W^{1,\infty}(C_{\rho_k}(x_0),\Iq)$ such that:
\begin{itemize}
\item[(a)] $w_k=\sum_{j=1}^J\a{w^j_k}$ with $w^j_k \in W^{1,\infty}(C_{\rho_k}(x_0),\I{q_j})$,
$\norm{\cG(w_k,u(x_0))}{L^\infty(C_{\rho_k}(x_0))}=o(1)$ and
$\cG(w_k(x),u(x_0))^2=\sum_{j=1}^J\cG(w_k^j(x),q_j\a{a_j})^2$ for every $x\in C_{\rho_k}(x_0)$;
\item[(b)] $\fint_{C_{\rho_k}(x_0)}\cG^p(w_{k},T_{x_0}u)=o(\rho_k^p)$;
\item[(c)] $\lim_{k\uparrow+\infty}\fint_{C_{\rho_k}(x_0)}f\big(x_0,u(x_0),Dw_k\big)=\frac{d\mu}{d\cL^m}(x_0)$.
\end{itemize}
\end{claim}
\begin{proof}
We choose radii $\rho_k$ which satisfy the following conditions:
\begin{gather}
\sup_k \fint_{C_{\rho_k}(x_0)}\big(\ph(|u_k|^q)+\ph(|Du_{k}|^p)\big)<+\infty,\label{e:rad1}\\
\fint_{C_{\rho_k}(x_0)}f\big(x,u_k,Du_k\big)\to\frac{d\mu}{d\cL^m}(x_0),\label{e:rad2}\\
\fint_{C_{\rho_k}(x_0)}\cG^p(u_{k},u)=o(\rho_k^p)\quad
\text{and}
\quad\fint_{C_{\rho_k}(x_0)}\cG^p(u_{k},T_{x_0}u)=o(\rho_k^p).\label{e:rad3}
\end{gather}
As for \eqref{e:rad1} and \eqref{e:rad2},
since 
$$
\left(\ph(|u_k|^q)+\ph(|Du_k|^p)\right)\cL^m\res\Omega\weaks\nu
\qquad\mbox{and}\qquad
f(x,u_k,Du_k)\cL^m\res\Omega\weaks\mu\, ,
$$ 
we only need to check that $\nu(\de C_{\rho_k}(x_0))=\mu(\de C_{\rho_k}(x_0))=0$
(see for instance Proposition 2.7 of \cite{DeLBook}).
Fixed such radii, for every $k$ we can choose a term in the 
sequence $(u_k)$ in such a way that the first half of
\eqref{e:rad3} holds (because of the strong convergence 
of $(u_k)$ to $u$): the second half is, hence, consequence of \eqref{e:caldzyg}.

Set $r_k=2\,|Du|(x_0)\,\rho_k$ and consider the retraction maps $\vartheta_k:\Iq\to \overline {B}_{r_k}(u(x_0))\subset \Iq$
constructed in \cite[Lemma 3.7]{DLSp1} (note that for $k$ sufficiently large, these maps are well defined).
The functions $w_k:=\vartheta_k\circ u_k$ satisfy the conclusions of the claim.

Indeed, since $\vartheta_k$ takes values in $\overline {B}_{r_k}(u(x_0))\subset \Iq$ and $r_k\to0$, (a) follows straightforwardly.
As for (b), the choice of $r_k$ implies that
$\vartheta_k\circ T_{x_0}u=T_{x_0}u$ on $C_{\rho_k}(x_0)$,
because
\begin{equation}\label{e:rk}
\cG(T_{x_0}u(x),u(x_0))\leq |Du(x_0)|\,|x-x_0|\leq
|Du(x_0)|\,\rho_k=\frac{r_k}{2}.
\end{equation}
Hence, being $\Lip(\vartheta_k)\leq 1$, from \eqref{e:rad3} we conclude
$$
\fint_{C_{\rho_k}(x_0)}\cG^p(w_{k},T_{x_0}u)
=\fint_{C_{\rho_k}(x_0)}\cG^p(\vartheta_k\circ u_{k},\vartheta_k\circ T_{x_0}u)
\leq \fint_{C_{\rho_k}(x_0)}\cG^p(u_{k},T_{x_0}u)=o(\rho_k^p).
$$
To prove (c), set $A_k=\big\{w_k\neq u_k\big\}=\{\cG(u_k,u(x_0))>r_k\}$ and note that,
by Chebychev's inequality, we have
\begin{align*}
r_k^p\,\cL^m(A_k) &\leq \int_{A_k}\cG^p(u_k,u(x_0))\leq
2^{p-1}\int_{A_k}\cG^p(u_k,T_{x_0}u)+2^{p-1}\int_{A_k}\cG^p(T_{x_0}u,u(x_0))\notag\\
& \stackrel{\eqref{e:rad3},\,\eqref{e:rk}}{\leq} o(\rho_k^{m+p})+\frac{r_k^p}{2}\,\cL^m(A_k),
\end{align*}
which in turn implies
\begin{equation}\label{e:Ak}
\cL^m(A_k)=o(\rho_k^m).
\end{equation}
Using Lemma~\ref{l:equi mean}, we prove that
\begin{equation}\label{e:f0 fw}
\lim_{k\to+\infty}\left(\fint_{C_{\rho_k}(x_0)}f\left(x_0,u(x_0),Dw_k\right)
-\fint_{C_{\rho_k}(x_0)}f\left(x,w_k,Dw_k\right)\right)=0.
\end{equation}
Indeed, for every $t>0$,
\begin{align}\label{e:f0 fwbis}
\lefteqn{\left|\fint_{C_{\rho_k}(x_0)}f\left(x_0,u(x_0),Dw_k\right)-\fint_{C_{\rho_k}(x_0)}f\left(x,w_k,Dw_k\right)\right|}\notag\\
&\leq\rho_k^{-m}\int_{C_{\rho_k}(x_0)\cap \{|Dw_k|\geq t\}}\Big(f\left(x_0,u(x_0),Dw_k\right)+f\left(x,w_k,Dw_k\right)\Big)\notag\\
&+\rho_k^{-m}\int_{C_{\rho_k}(x_0)\cap \{|Dw_k|< t\}}|f\left(x_0,u(x_0),Dw_k\right)-f\left(x,w_k,Dw_k\right)|\notag\\
&\leq \sup_k\frac{C}{\rho_k^{m}}\int_{C_{\rho_k}(x_0)\cap \{|Dw_k|\geq t\}}\big(1+|w_k|^q+|Dw_k|^p\big)+
\omega_{f,t}(\rho_k+\norm{\cG(w_k,u(x_0)}{L^{\infty}}),
\end{align}
where $\omega_{f,t}$ is a modulus of continuity for $f$ restricted to the compact set
$\overline{C}_{\rho_1}(x_0)\times \overline{B}_{|u(x_0)|+1}\times \overline B_t\subset\Omega \times (\R{n})^Q\times(\R{m+n})^Q$.
To fully justify the last inequality we remark that we choose the same order of the gradients in both integrands so that the order
for $u(x_0)$ and for $w_k$ is the one giving the $L^\infty$ distance between them.
Then, \eqref{e:f0 fw} follows by passing to the limit in \eqref{e:f0 fwbis}
first as $k\to+\infty$ and then as $t\to+\infty$ thanks to \eqref{e:eq mean1} in Lemma \ref{l:equi mean} applied to $1+|w_k|^q$ (which is equi-bounded in $L^\infty(C_{\rho_k}(x_0))$ and, hence, equi-integrable) and to $|Dw_k|^p$.

Thus, in order to show item (c), it suffices to prove
\begin{equation}\label{e:fw fu}
\lim_{k\to+\infty}\left(\fint_{C_{\rho_k}(x_0)}f\left(x,u_k,Du_k\right)
-\fint_{C_{\rho_k}(x_0)}f\left(x,w_k,Dw_k\right)\right)=0\, .
\end{equation}
By the definition of $A_k$, we have
\begin{align*}\label{e:fw fubis}
\lefteqn{\left|\fint_{C_{\rho_k}(x_0)}f\left(x,u_k,Du_k\right)-\fint_{C_{\rho_k}(x_0)}f\left(x,w_k,Dw_k\right)\right|}\notag\\
&\leq\rho_k^{-m}\int_{A_k}\Big(f\left(x,u_k,Du_k\right)+f\left(x,w_k,Dw_k\right)\Big)\notag\\
&\leq \frac{C}{\rho_k^{m}}\int_{A_k}\big(1+|w_k|^q+|u_k|^q+|Dw_k|^p+|Du_k|^p\big).
\end{align*}
Hence, by the equi-integrability of $u_k$, $w_k$ and their gradients, and by \eqref{e:Ak},
we can conclude from \eqref{e:eq mean2} of Lemma \ref{l:equi mean}
\end{proof}

Using Claim \ref{cl:1}, we can now ``blow-up'' the functions 
$w_k$ and conclude the proof of \eqref{e:sc density}.
More precisely we will show:

\begin{claim}\label{cl:2}
For every $\gamma>0$, there exist $(z_k)\subset W^{1,\infty}(C_1,\Iq)$ 
such that $z_k\vert_{\de C_1}=T_{x_0}u\vert_{\de C_1}$ for every $k$ and
\begin{equation}\label{e:sc tilde}
\limsup_{k\to+\infty}\int_{C_1}f\big(x_0,u(x_0),D z_k\big)\leq \frac{d\mu}{d\cL^m}(x_0)+\gamma.
\end{equation}
\end{claim}

Assuming the claim and testing the definition of quasiconvexity of $f(x_0,\cdot,\cdot)$
through the $z_k$'s, by \eqref{e:sc tilde}, we get
$$
f\big(x_0,u(x_0),Du(x_0)\big)\leq\limsup_{k\to+\infty}\int_{C_1}f\big(x_0,u(x_0),Dz_k\big)\leq\frac{d \mu}{d\cL^m}(x_0)+\gamma,
$$
which implies \eqref{e:sc density} by letting $\gamma\downarrow0$ and concludes the proof.

\begin{proof}[Proof of Claim \ref{cl:2}]
We consider the functions $w_k$ of Claim \ref{cl:1} and, since they have full multiplicity at $x_0$, we can blow-up.
Let $\zeta_k:=\sum_{j=1}^J\a{\zeta_k^j}$ with the maps $\zeta_k^j\in W^{1,\infty}(C_1,\I{q_j})$ defined by
$\zeta_k^j(y):=\tau_{-a_j}\big(\rho_k^{-1}\,\tau_{a_j}(w_k^j)(x_0+\rho_k\cdot)\big)(y)$, with $\tau_{-a_j}$ defined in \eqref{e:translation}.
Clearly, a simple change of variables gives
\begin{equation}\label{e:blow1}
\zeta_k^j\to q_j\a{a_j+L_j\cdot}\quad \text{in}\quad L^{p}(C_1,\I{q_j})
\end{equation}
and, by Claim $1$ (c),
\begin{equation}\label{e:blow2}
\lim_{k\to+\infty}\int_{C_{1}}f\big(x_0,u(x_0),D\zeta_k\big)
=\frac{d\mu}{d\cL^m}(x_0).
\end{equation}
Now, we modify the sequence $(\zeta_k)$
into a new sequence $(z_k)$
in order to satisfy the boundary conditions and \eqref{e:sc tilde}.
For every $\delta>0$, we find $r\in (1-\delta,1)$ such that
\begin{equation}\label{e:r}
\liminf_{k\to+\infty}\int_{\de C_r}|D \zeta_k|^p\leq \frac{C}{\delta}\quad\text{and}\quad
\lim_{k\to+\infty}\int_{\de C_r}\cG^p(\zeta_k,T_{x_0}u)=0.
\end{equation}
Indeed, by using Fatou's lemma, we have
\begin{align*}
\int_{1-\delta}^1\liminf_{k\to+\infty}\int_{\de C_s}|D\zeta_k|^pds\leq 
\liminf_{k\to+\infty}\int_{C_1\setminus C_{1-\delta}}|D\zeta_k|^p\leq C,\\
\int_{1-\delta}^1\lim_{k\to+\infty}\int_{\de C_s}\cG^p(\zeta_k,T_{x_0}u)ds\leq 
\liminf_{k\to+\infty}\int_{C_1\setminus C_{1-\delta}}\cG^p(\zeta_k,T_{x_0}u)\stackrel{\eqref{e:blow1}}{=}0,
\end{align*}
which together with the mean value theorem gives \eqref{e:r}.
Then we fix $\eps>0$ such that $r(1+\eps)<1$ and we apply the
interpolation result \cite[Lemma 2.15]{DLSp1} to infer the existence
of a function $z_k\in W^{1,\infty}(C_1,\Iq)$ such that
$z_k\vert_{C_r}=\zeta_k\vert_{C_r}$,
$z_k\vert_{C_1\setminus C_{r(1+\eps)}}=T_{x_0}u\vert_{C_1\setminus C_{r(1+\eps)}}$
and
\begin{align}\label{e:enrg raccordo}
\int_{C_{r(1+\eps)\setminus C_r}}|Dz_k|^p&\leq
C\,\eps\,r\left(\int_{\de C_r}|D\zeta_k|^p+
\int_{\de C_r}|DT_{x_0}u|^p\right)+\frac{C}{\eps\,r}
\int_{\de C_r}\cG^p(\zeta_k,T_{x_0}u)\notag\\
&\leq C\,\eps (1+\delta^{-1})+\frac{C}{\eps}\int_{\de C_r}\cG^p(\zeta_k,T_{x_0}u).
\end{align}
Therefore, by \eqref{e:enrg raccordo}, we infer
\begin{align*}
\int_{C_1}f\big(x_0,u(x_0),Dz_k\big)={}&
\int_{C_r}f\big(x_0,u(x_0),D\zeta_k\big)\\
&+\int_{C_{r(1+\eps)}\setminus C_r}\hspace{-0.5cm}f\big(x_0,u(x_0),Dz_k\big)
+\int_{C_1\setminus C_{r(1+\eps)}}\hspace{-0.5cm}
f\big(x_0,u(x_0),Du(x_0)\big)\\
\leq{}& \int_{C_1}f\big(x_0,u(x_0),D\zeta_k\big)+
C\,\eps (1+\delta^{-1})+\frac{C}{\eps}\int_{\de C_r}\cG^p(\zeta_k,T_{x_0}u)+C\delta.
\end{align*}
Choosing $\delta>0$ and  $\eps>0$ such that $C\,\eps
(1+\delta^{-1})+C\delta\leq \gamma$, 
and taking the superior limit as $k$ goes to $+\infty$ in the latter
inequality, we get \eqref{e:sc tilde} thanks to \eqref{e:blow2}
and \eqref{e:r}.
\end{proof}

\subsection{Necessity of quasiconvexity}
We now prove that, if $F$ is weak$^*$-$W^{1,\infty}$ lower semicontinuous, then $f(x_0,\cdot,\cdot)$ is $Q$-quasiconvex for every $x_0\in\Omega$.
Without loss of generality, assume $x_0=0$ and fix an affine $Q$-function $u$ and functions $w^j$ as in Definition \ref{d:qc}.
Set $z^{j}(y):=\sum_{i=1}^{q_j}\a{(w^j(y))_i-a_j-L_j\cdot y}$, so that $z^j\vert_{\de C_1}=q_j\a{0}$, and extend it by $C_1$-periodicity.

We consider $v_k^j(y)=\sum_{i=1}^{q_j}\a{k^{-1}(z^j(ky))_i+a_j+L_j\cdot y}$
and, for every $r>0$ such that $C_r\subseteq \Omega$, we define
$u_{k,r}(x)=\sum_{j=1}^J\tau_{(r-1)a_j}\left(r\,v_k^j\left(r^{-1}x\right)\right)$.
Note that:
\begin{itemize}
\item[(a)] for every $r$, $u_{k,r}\to u$ in $L^\infty(C_r,\Iq)$ as $k\to+\infty$;
\item[(b)] $u_{k,r}\vert_{\de C_r}=u\vert_{\de C_r}$ for every $k$ and $r$;
\item[(c)] for every $k$, $u_{k,r}(0)=\sum_{j=1}^J\tau_{-a_j}\left(r/k\,z^j(0)\right)\to u(0)$ as $r\to0$;
\item[(d)] for every $r$,
$\sup_{k}\norm{|Du_{k,r}|}{L^\infty(C_r)}<+\infty$, since
$$
|Du_{k,r}|^2(x)=\sum_{j=1}^J|Dv_k^j|^2\left(r^{-1}x\right)=
\sum_{j=1}^J\sum_{i=1}^{q_j}\left|Dz^j_i\left(k\,r^{-1}x\right)+L_j\right|^2.
$$
\end{itemize}
From (a) and (d) it follows that, for every $r$,
$u_{k,r}\weaks u$ in $W^{1,\infty}(C_r,\Iq)$ as $k\to+\infty$.
Then, by (b), setting $u_{k,r}=u$ on $\Omega\setminus C_r$,
the lower semicontinuity of $F$ implies that
\begin{equation}\label{e:loc}
F\big(u,C_r\big):=\int_{C_r}f\big(x, u, Du\big)\leq
\liminf_{k\to+\infty}F\big(u_{k,r},C_r\big).
\end{equation}
By the definition of $u_{k,r}$, changing the variables in \eqref{e:loc},
we get
\begin{multline}\label{e:loc2}
\int_{C_1}f\big(ry, \underbrace{a_1+r\,L_1\cdot y}_{q_1},
\ldots,\underbrace{a_J+r\,L_J\cdot y}_{q_J},
L_1,\ldots,L_J\big)dy\\
\leq\liminf_{k\to\infty}
\int_{C_1}f\big(ry, \tau_{(r-1)a_1}(r\,v_k^1(y)),\ldots,
\tau_{(r-1)a_J}(r\,v_k^J(y)),
Dv_k^1(y),\ldots,Dv_k^J(y)\big)dy.
\end{multline}
Noting that $\tau_{(r-1)a_j}(r\,v_k^j(y))\to q_j\a{a_j}$ in $L^\infty(C_1,\I{q_j})$
as $r$ tends to $0$
and $Dv_k^j(y)=\tau_{-L_j}(Dz^j(ky))$, \eqref{e:loc2} leads to
\begin{multline}\label{e:loc3}
f\big(0,\underbrace{a_1,\ldots,a_1}_{q_1},
\ldots,\underbrace{a_J,\ldots,a_J}_{q_J},
L_1,\ldots,L_J\big)\\
\leq\liminf_{k\to\infty}
\int_{C_1}f\big(0, \underbrace{a_1,\ldots,a_1}_{q_1},
\ldots,\underbrace{a_J,\ldots,a_J}_{q_J},
\tau_{-L_1}(Dz^1(ky)),\ldots,\tau_{-L_J}(Dz^J(ky))\big)dy.
\end{multline}
Using the periodicity of $z^j$, the integral on the right hand side of \eqref{e:loc3} equals
$$
\int_{C_1}f\big(0, \underbrace{a_1,\ldots,a_1}_{q_1},
\ldots,\underbrace{a_J,\ldots,a_J}_{q_J},
\tau_{-L_1}(Dz^1(y)),\ldots,\tau_{-L_J}(Dz^J(y))\big)dy.
$$
Since $\tau_{-L_j}(Dz^j)=Dw^j$, we conclude \eqref{e:qc}.

\section{Polyconvexity}\label{s:poly}
In this section we prove Theorem \ref{t:pc=>qc} and show the semicontinuity of Almgren's Dirichlet energy and Mattila's quadratic energies.
Recall the notation for multi-indices and minors $M_{\alpha,\beta}$ introduced in Section \ref{s:prel}.

\begin{definition}\label{d:poliaffine}
A map $P: \R{n\times m} \to \R{}$ is polyaffine if there are constants $c_0, c^l_{\alpha\beta}$,
for $l\in \{1, \ldots, N\}$ and $\alpha, \beta$ multi-indices, such that
\begin{equation}\label{e:poliaffine}
P (A) = c_0 + \sum_{l=1}^N\; \sum_{|\alpha|=|\beta|=l}
c^l_{\alpha \beta}\, M_{\alpha\beta} (A)= c_0+\langle \zeta,M(A)\rangle,
\end{equation}
where $\zeta\in\R{\tau(m,n)}$ is the vector whose entries are the $c^l_{\alpha\beta}$'s and $M(A)$ is the vector of all minors.
\end{definition}

It is possible to represent polyconvex functions as supremum of a family of polyaffine functions
retaining some symmetries from the invariance of $f$ under the action of permutations.

\begin{propos}\label{p:chrctrztn}
Let $f$ be a $Q$-integrand, then the following are equivalent:
\begin{itemize}
\item[(i)] $f$ is a polyconvex  $Q$-integrand,
\item[(ii)] for every choice of vectors $a_1, \ldots, a_Q \in \R{n}$ and matrices $A_1, \ldots A_Q \in \R{n\times m}$,
with $A_i=A_j$ if $a_i=a_j$, there exist polyaffine functions $P_j:\R{n\times m}\to\R{}$, with $P_i=P_j$ if $a_i=a_j$, such that
\begin{equation}\label{e:touch}
f \big(a_1, \ldots, a_Q,A_1, \ldots, A_Q\big)= \sum_{j=1}^Q P_j(A_j),
\end{equation}
and
\begin{equation}\label{e:above}
f \big(a_1, \ldots, a_Q,L_1, \ldots, L_Q\big)\geq
\sum_{j=1}^Q P_j(L_j)\quad\text{for every } L_1, \ldots, L_Q\in \R{n\times m}.
\end{equation}
\end{itemize}
\end{propos}
\begin{proof} \noindent\textit{(i)$\Rightarrow$(ii).} 
Let $g$ be a function representing $f$ according to Definition~\ref{d:pc}.
Convexity of the subdifferential of $g(a_1, \ldots, a_Q,\cdot)$, condition \eqref{e:char-cvx} and the invariance of $f$ under the
action of permutations yield that there exists $\zeta\in\de g\big(a_1, \ldots, a_Q,M(A_1), \ldots, M(A_Q)\big)$,
with $\zeta_i=\zeta_j$ if $a_i=a_j$, such that for every $X\in (\R{\tau(m,n)})^Q$ we have
\begin{equation}\label{e:hcxvty}
g(a_1, \ldots, a_Q,X_1,\ldots,X_Q)\geq 
g\big(a_1, \ldots, a_Q,M(A_1), \ldots, M(A_Q)\big)
+\sum_{j=1}^Q
\langle \zeta_j,X_j-M(A_j)\rangle.
\end{equation}
Hence, the maps $P_j:\R{n\times m}\to\R{}$ given by
\begin{equation}\label{e:P}
P_j(L):=Q^{-1} g\big(a_1, \ldots, a_Q,M(A_1), \ldots, M(A_Q)\big)+\langle \zeta_j,M(L)-M(A_j)\rangle
\end{equation}
are polyaffine and such that \eqref{e:touch} and \eqref{e:above} follow.

\noindent\textit{(ii)$\Rightarrow$(i).}
By \eqref{e:touch} and \eqref{e:above}, there exists $\zeta_j$, satisfying $\zeta_i=\zeta_j$ if $a_i=a_j$, such that
\begin{eqnarray}\label{e:above2}
f \big(a_1, \ldots, a_Q, L_1, \ldots, L_Q\big)\geq
f \big(a_1, \ldots, a_Q,A_1,\ldots,A_Q\big)
+\sum_{j=1}^Q\langle\zeta_j, M(L_j)-M(A_j)\rangle.
\end{eqnarray}
Then setting,
\begin{multline*}
g\big(a_1, \ldots, a_Q,X_1, \ldots, X_Q\big):=
\sup\Bigl\{ f\big(a_1, \ldots, a_Q,A_1, \ldots, A_Q\big)
+\sum_{j=1}^Q\langle\zeta_j, X_j-M(A_j)\rangle\Bigr\}
\end{multline*}
where the supremum is taken over all $A_1, \ldots, A_Q\in \R{n\times m}$ with $A_i=A_j$ if $a_i=a_j$,
it follows clearly that $g\big(a_1, \ldots, a_Q,\cdot\big)$ is a convex function and \eqref{e:char-cvx} holds thanks to \eqref{e:above2}.
In turn, these remarks and the equality  
$\textrm{co}\big((M(\R{n\times m}))^Q\big)=(\R{\tau(m,n)})^Q$ 
imply that $g\big(a_1, \ldots, a_Q,\cdot\big)$ is everywhere finite.
\end{proof}

We are now ready for the proof of Theorem \ref{t:pc=>qc}.
\begin{proof}[Proof of Theorem \ref{t:pc=>qc}]
Assume that $f$ is a polyconvex $Q$-integrand and consider $a_j, L_j$ and $w^j$ as in Definition \ref{d:qc}.
Corresponding to this choice, by Proposition \ref{p:chrctrztn}, there exist 
polyaffine functions $P_j$ satisfying \eqref{e:touch} and \eqref{e:above},
which read as
\begin{equation}\label{e:touch bis}
f \big(\underbrace{a_1, \ldots, a_1}_{q_1},\ldots, \underbrace{a_J, \ldots, a_J}_{q_J},
\underbrace{L_1, \ldots, L_1}_{q_1},\ldots\underbrace{L_J, \ldots, L_J}_{q_J}\big)
= \sum_{j=1}^J q_j P_j (L_j)
\end{equation}
and, for every $B_1, \ldots, B_Q\in \R{m\times n}$,
\begin{equation}\label{e:above bis}
f (\underbrace{a_1, \ldots, a_1}_{q_1},
\ldots, \underbrace{a_J, \ldots, a_J}_{q_J},
B_1, \ldots, B_Q)\geq
\sum_{j=1}^J \left\{\sum_{i=\sum_{l<j} q_l +1}^{\sum_{l\leq j}
q_l} P_j (B_i)\right\}.
\end{equation}

To prove the theorem it is enough to show that
\begin{equation}\label{e:scpoly}
\sum_{j=1}^Jq_j\,P_j(L_j)= \int_{C_1}
\sum_{j=1}^J \sum_{i=1}^{q_j} P_j (Dw^j_i).
\end{equation}
Indeed, then the quasiconvexity of $f$ follows easily from
\begin{multline*}
f \big(\underbrace{a_1, \ldots, a_1}_{q_1},
\ldots, \underbrace{a_J, \ldots, a_J}_{q_J},
\underbrace{L_1, \ldots, L_1}_{q_1},
\ldots\underbrace{L_J, \ldots, L_J}_{q_J}\big)\stackrel{\eqref{e:touch}}{=}
\sum_{j=1}^Jq_j\,P_j(L_j)\\
\stackrel{\eqref{e:scpoly}}{=} \int_{C_1}
\sum_{j=1}^J \sum_{i=1}^{q_j} P_j (Dw^j_i)\stackrel{\eqref{e:above}}{\leq}
\int_{C_1}f \big(\underbrace{a_1, \ldots, a_1}_{q_1},
\ldots, \underbrace{a_J, \ldots, a_J}_{q_J},
Dw^1,\ldots,Dw^J\big).
\end{multline*}

To prove \eqref{e:scpoly}, consider the current $T_{w^j,C_1}$
associated to the graph of the $q_j$-valued map $w^j$ and note that, by definition \eqref{e:Tu minori}, 
for the exact, constant coefficient $m$-form
$d\omega^j=c_0^j\,dx+\sum_{l=1}^N\sum_{|\alpha|=|\beta|=l} \sigma_\alpha\, c^{j,l}_{\alpha\beta}\, dx_{\bar\alpha}\wedge dy_\beta$,
it holds
\begin{equation}\label{e:int parti}
\int_{C_1} \sum_{i=1}^{q_j} P_j (Dw^j_i)=\la T_{w^j,C_1},d\omega^j\ra,
\end{equation}
where $P_j (A) = c_0^j + \sum_{l=1}^N\sum_{|\alpha|=|\beta|=l}c^{j,l}_{\alpha\beta}\, M_{\alpha\beta} (A)$.

Since $u\vert_{\de C_1}=w\vert_{\de C_1}$, from Theorem \ref{t:de Tf} it follows that $\de T_{w,C_1}=\de T_{u, C_1}$.
Then, \eqref{e:scpoly} is an easy consequence of \eqref{e:int parti}: for $u^j(x)=q_j\a{a_j+L_j\cdot x}$, one has, indeed,
\begin{align*}
\sum_{j=1}^Jq_j\,P_j(L_j) &=
\int_{C_1} \sum_{j=1}^J\sum_{i=1}^{q_j}P_j (Du^j_i)=\sum_{j=1}^J\la T_{u^j,C_1},d\omega^j\ra=\sum_{j=1}^J\la \de T_{u^j,C_1},\omega^j\ra\\
&=\sum_{j=1}^J\la \de T_{w^j,C_1},\omega^j\ra=
\sum_{j=1}^J\la T_{w^j,C_1},d\omega^j\ra=
\int_{C_1}\sum_{j=1}^J \sum_{i=1}^{q_j} P_j (Dw^j_i).
\end{align*}
This finishes the proof.
\end{proof}

Explicit examples of polyconvex functions are collected below (the elementary proof is left to the reader).
\begin{propos}\label{p:ex}
The following class of functions are polyconvex $Q$-integrands:
\begin{itemize}
\item[(a)] $f (a_1, \ldots, a_Q, L_1, \ldots, L_Q):=g\big(\cG(L,Q\a{0})\big)$
with $g:\R{}\to\R{}$ convex and increasing;
\item[(b)] $f (a_1, \ldots, a_Q, L_1, \ldots, L_Q):=\sum_{i,j=1}^Qg(L_i-L_j)$ 
with $g:\R{n\times m}\to\R{}$ convex;
\item[(c)] $f (a_1, \ldots, a_Q, L_1, \ldots, L_Q):=
\sum_{i=1}^Q g (a_i, L_i)$ with $g: \R{m}\times \R{n\times m}\to \R{}$ measurable and polyconvex.
\end{itemize}
\end{propos}

\begin{remark}\label{r: mat3}
Consider as in Remark \ref{r: mat} a linear symmetric map $\R{n\times m}\ni M\mapsto A\, M\in\R{n\times m}$.
As it is well-known, for classical single valued functions the functional
\begin{equation*}
\int\langle A\,Df,Df\rangle
\end{equation*}
is quasiconvex if and only if it is rank-$1$ convex.
If $\min\{m,n\}\leq 2$, quasiconvexity is equivalent to polyconvexity as well (see \cite{Te}).
Hence, in this case, by Theorem \ref{t:pc=>qc}, every $1$-semielliptic integrand is quasiconvex and therefore $Q$-semielliptic.

We stress that for $\min\{m,n\}\geq3$ there exist $1$-semielliptic integrands which are not polyconvex (see always \cite{Te}).
\end{remark}

\appendix

\section{Equi-integrability}
Let us first recall some definitions and introduce some notation.
As usual, in the following $\Omega\subset\R{m}$ denotes a Lipschitz set with finite measure.

\begin{definition}\label{d:equi-int}
A sequence $(g_k)$ in $L^1(\Omega)$ is \emph{equi-integrable} if one of the following equivalent conditions holds:
\begin{itemize}
 \item[(a)] for every $\eps>0$ there exists $\delta>0$ such that, for every $\cL^m$-measurable set $E\subseteq\Omega$ with
$\cL^m(E)\leq\delta$, we have $\sup_k\int_{E}|g_k|\leq\eps$;
 \item[(b)] the distribution functions $\varphi_k(t):=\int_{\{|g_k|\geq t\}}|g_k|$ satisfy $\lim_{t\to+\infty}\sup_k\varphi_k(t)=0$;
 \item[(c)] (De la Vall\'ee Poissin's criterion) if there exists a Borel function $\varphi:[0,+\infty)\to[0,+\infty]$ 
such that
\begin{equation}
 \lim_{t\to+\infty}\frac{\varphi(t)}t=+\infty\,\text{ and }\, \sup_k\int_{\Om}\varphi(|g_k|)dx<+\infty.
\end{equation}
\end{itemize}
\end{definition}
Note that, since $\Omega$ has finite measure, an equi-integrable sequence is also equi-bounded.
We prove now Chacon's biting lemma.

\begin{lemma}\label{l:BM}
Let $(g_k)$ be a bounded sequence in $L^1(\Omega)$. Then, there
exist a subsequence $(k_j)$ and a sequence $(t_j)\subset[0,+\infty)$
with $t_j\to+\infty$ such that $(g_{k_j}\vee(-t_j)\wedge t_j)$ is equi-integrable.
\end{lemma}

\begin{proof}
Without loss of generality, assume $g_k\geq0$ and consider for every $j\in\N$ the functions $h_k^j:=g_k\wedge j$.
Since, for every $j$, $(h_k^j)_k$ is equi-bounded in $L^\infty$, up to passing to a subsequence (not relabeled),
there exists the $L^\infty$ weak* limit $f_j$ of $h_k^j$ for every $j$.
Clearly the limits $f_j$ have the following properties:
\begin{itemize}
\item[(a)] $f_j\leq f_{j+1}$ for every $j$ (since $h_k^j\leq h_k^{j+1}$ for every $k$);
\item[(b)] $\norm{f_j}{L^1}=\lim_k\norm{h_k^j}{L^1}$;
\item[(c)] $\sup_j \norm{f_j}{L^1}= \sup_j \lim_k\norm{h_k^j}{L^1}\leq \sup_k\norm{g_k}{L^1}<+\infty$.
\end{itemize}
By the Lebesgue monotone convergence theorem, (a) and (c), it follows that $(f_j)$ converges in $L^1$ to a function $f$.
Moreover, from (b), for every $j$ we can find a $k_j$ such that 
\begin{equation}\label{e:mon2}
\left|\int h_{k_j}^j-\int f_j\right|\leq j^{-1}.
\end{equation}
We claim that $h_{k_j}^j=g_{k_j}\wedge j$ fulfills the conclusion of the lemma (with $t_j=j$).
To see this, it is enough to show that $h_{k_j}^j$ weakly converges to $f$ in $L^1$, from which the equi-integrability follows.
Let $a\in L^\infty$ be a test function. Since $h_{k_j}^l\leq h_{k_j}^j$ for $l\leq j$, we have that
\begin{equation}\label{e:mon}
\int\big(\norm{a}{L^\infty}-a\big)h_{k_j}^l\leq \int\big(\norm{a}{L^\infty}-a\big)h_{k_j}^j.
\end{equation}
Taking the limit as $j$ goes to infinity in \eqref{e:mon}, we obtain (by $h_{k_j}^l\stackrel{w^*\hbox{-}L^\infty}{\to} f_l$
and \eqref{e:mon2}) 
\begin{equation*}
\int\big(\norm{a}{L^\infty}-a\big)f_l \leq \norm{a}{L^\infty}\int f-\limsup_j\int a\, h_{k_j}^j.
\end{equation*}
From which, passing to the limit in $l$,  we conclude since $f_l\stackrel{L^1}{\to} f$
\begin{equation}\label{e:mon3}
\limsup_j\int a\, h_{k_j}^j\leq \int a f.
\end{equation}
Using $-a$ in place of $a$, one obtains as well the inequality
\begin{equation}\label{e:mon4}
\int a f\leq \liminf_j\int a\, h_{k_j}^j.
\end{equation}
\eqref{e:mon3} and \eqref{e:mon4} together concludes the proof of the weak convergence of $h_{k_j}^j$ to $f$ in $L^1$.
\end{proof}

Next we show that concentration effects for critical Sobolev embedding
do not show up if equi-integrability of functions and gradients is assumed.
\begin{lemma}\label{l:equipstar}
Let $p\in[1,m)$ and $(g_k)\subset W^{1,p}(\Om)$ be such that 
$(|g_k|^p)$ and $(|\nabla g_k|^p)$ are both equi-integrable, then 
$(|g_k|^{p^*})$ is equi-integrable as well.
\end{lemma}
\begin{proof}
Since $(g_k)$ is bounded in $W^{1,p}(\Om)$, Chebychev's inequality implies 
\begin{equation}\label{e:suplevel1}
\sup_jj^p\cL^m(\{|g_k|>j\})\leq C<+\infty.
\end{equation}
For every fixed $j\in\N{}$, consider the sequence $g_k^j:=g_k-(g_k\vee(-j)\wedge j)$.
Then, $(g_k^j)\subset W^{1,p}(\Om)$ and
$\nabla g_k^j=\nabla g_k$ in $\{|g_k|>j\}$ and $\nabla g_k^j=0$ otherwise.
The Sobolev embedding yields
\begin{equation}\label{e:suplevel2}
\|g_k^j\|_{L^{p^*}(\Omega)}^p\leq c\|g_k^j\|_{W^{1,p}(\Omega)}^p\leq 
c\int_{\{|g_k|>j\}}\big(|g_k|^p+|\nabla g_k|^p\big)dx.
\end{equation}
Therefore, the equi-integrability assumptions and \eqref{e:suplevel1} imply that for every $\eps>0$ there exists $j_\eps\in\N{}$ 
such that for every $j\geq j_\eps$
\begin{equation}\label{e:norma}
\sup_k\|g_k^j\|_{L^{p^*}(\Omega)}\leq\eps/2.
\end{equation}
Let $\delta>0$ and consider a generic $\cL^m$-measurable sets $E\subseteq\Om$ with $\cL^m(E)\leq\delta$.
Then, since we have 
$$
\|g_k\|_{L^{p^*}(E)}\leq\|g_k-g_k^{j_\eps}\|_{L^{p^*}(E)}+\|g_k^{j_\eps}\|_{L^{p^*}(E)}\leq
{j_\eps}\,(\cL^m(E))^{1/p^*}+\|g_k^{j_\eps}\|_{L^{p^*}(\Omega)},
$$
by \eqref{e:norma}, to conclude it suffices to choose $\delta$ such that $j_\eps\delta^{1/p^*}\leq\eps/2$.
\end{proof}

\bibliographystyle{plain}
\bibliography{reference-semi}
\end{document}